\documentclass[12pt]{article}
\usepackage[centertags]{amsmath}     
\usepackage{amssymb}
\usepackage{amsthm}
\usepackage{verbatim}                
\usepackage[dvips]{graphics}
\usepackage{enumerate}

\numberwithin{equation}{section}
\newtheorem{prop}[equation]{Proposition}
\newtheorem{lem}[equation]{Lemma}
\newtheorem{lemma}[equation]{Lemma}

\newtheorem{cor}[equation]{Corollary}
\newtheorem{thm}[equation]{Theorem}

\newtheorem{claim}[equation]{Claim}

\newtheorem*{thm1.3p}{Theorem 1.3'}
\newtheorem*{thm1.3pp}{Theorem 1.3''}

\newcommand{\me}{\mathrm{e}}

\newcommand{\eps}{\varepsilon}

\newcommand{\wtld}[1]{{\widetilde{#1}}}
\newcommand{\vphi}{\varphi}

\newcommand{\CC}{\mathbb{C}}

\newcommand{\Rb}{\mathbb{R}}
\newcommand{\Sb}{\mathbb{S}}  
\newcommand{\Tb}{\mathbb{T}}
\newcommand{\ZZ}{\mathbb{Z}}

\newcommand{\Al}{\mathcal{A}}

\newcommand{\Hl}{\mathcal{H}}

\newcommand{\DIV}{\mathrm{div}}

\newcommand{\Vol}{\mathrm{Vol}}

\newcommand{\grad}{\mathrm{grad}}

\newcommand{\dif}[1]{{\;d #1}}
\newcommand{\dx}{\;dx}

\newcommand{\Bhf}{\frac{1}{2}B}

\begin{document}
\title{The Volume of a Local Nodal Domain}
\author{Dan Mangoubi}
\date{}
\maketitle
\begin{abstract}
Let $M$ either be a closed real analytic Riemannian manifold or a
closed $C^\infty$-Riemannian surface. We estimate from below the
volume of a nodal domain component in an arbitrary ball, provided
that this component enters the ball deeply enough.
\end{abstract}
\section{Introduction}
\subsection{Main Results}
Let $(M, g)$ be a closed $C^\infty$-Riemannian manifold of dimension $n$.
Let $\Delta=-\DIV\circ\grad$ be the Laplace--Beltrami operator on $M$.
We consider the eigenvalue equation
\begin{equation}
  \Delta \vphi_\lambda = \lambda \vphi_\lambda
\end{equation}
For any $\lambda$-eigenfunction $\vphi_\lambda$
the null set $\{\vphi_\lambda=0\}$ is called the
$\vphi_\lambda$-\emph{nodal set}
and any connected component of the set $\{\vphi_\lambda\neq 0\}$
is called a $\lambda$-\emph{nodal domain}.

The Faber-Krahn Inequality (\cite{EgoKon}) shows that the volume of any
$\lambda$-nodal domain $\Al_\lambda$ is $\geq
C/(\sqrt{\lambda})^n$.
%
%
Donnelly and Fefferman initiated in~\cite{DonFef90} the study of a
local version of the Faber-Krahn inequality. Namely, they gave
a lower bound on the volume of \emph{local} nodal domains:
\begin{thm}[\cite{DonFef90, ChaMuc91, Lu93}]
 \label{thm:local-courant}
   Let $\vphi_\lambda$ be a $\lambda$-eigenfunction.
   Let $B$ be an arbitrary metric ball in $M$,
   and let $\Omega_\lambda$ be a connected component
   of $\{\vphi_\lambda\neq 0\}\cap B$.
   If $\Omega_\lambda\cap\Bhf\neq\emptyset$ then
     $$\frac{|\Omega_\lambda|}{|B|}\geq
     \frac{C_1}{(\sqrt{\lambda})^{\alpha(n)}(\log\lambda)^{4n}}\ ,$$
   where $\alpha(n)=4n^2+n/2$.
  \end{thm}
\noindent\textbf{Notations.} In the above theorem and throughout
the paper $rB$ denotes a concentric ball whose radius is $r$ times
that of $B$. $C_1, C_2, \ldots$ denote constants which depend only
on the metric $g$. The enumeration of constants is different in
each section.\vspace{1Ex}

Donnelly and Fefferman note that the estimates above are clearly
non-sharp and they conjecture sharp estimates. The aim of the
present paper is to give practically sharp estimates in the case
of real analytic metrics and in the case of $C^{\infty}$-surfaces.
We prove:
\begin{thm}
 \label{thm:local-courant-analytic}
   Let $(M, g)$ be a closed real analytic Riemannian manifold.
   Let $\vphi_\lambda$ be as above. Let $B\subseteq M$ be an
   arbitrary metric ball of radius $R$,
   and let $\Omega_\lambda$ be a connected component of
   $\{\vphi_\lambda\neq 0\}\cap B$. If 
$\Omega_\lambda\cap\Bhf\neq\emptyset$ then
     $$\frac{|\Omega_\lambda|}{|B|}\geq
     \frac{C_2}{(\sqrt{\lambda})^{2n-2}R'(\log\lambda)^{n-1}}\ ,$$
   where $R' =\max\{R, 1/\sqrt{\lambda}\}$.
\end{thm}

In particular, for $R<1/\sqrt{\lambda}$, we have
\begin{equation}
\label{local-courant-small} |\Omega_\lambda|/|B| \geq
C_2(\sqrt{\lambda})^{-(n-1)} (\log\lambda)^{-(n-1)}\ .
\end{equation}

To better understand Theorem~\ref{thm:local-courant-analytic} consider a harmonic function
$\vphi$ defined in the ball of radius $2$, $B_2\subseteq\Rb^n$.
Assume $\vphi(0)=0$, and suppose also that its growth $\beta$ is given
by $$\beta:= \log\frac{\sup_{B_1}|\vphi|}{\sup_{B_{1/2}}|\vphi|}\ .$$
We consider the positive and negative components of $\vphi$ in the unit ball
$B_1$. Take one positive component $\Omega$ of $\{\vphi>0\}\cap B_1$.
We show that if $\Omega\cap B_{1/2}\neq\emptyset$ then
\begin{equation}
\label{harmonic-vol}
  \Vol(\Omega)\geq \frac{C}{\beta^{n-1}}\ .
\end{equation}
In fact we prove a similar estimate for a small perturbation of harmonic functions,
and use a well known scaling argument to pass to the estimate for eigenfunctions
on the wavelegth $(1/\sqrt{\lambda})$ scale.

One of our main motivation to prove the estimate~(\ref{harmonic-vol})
besides the interest raised by Donnelly and Fefferman is the following
result in~\cite{NazPolSod05}:
\begin{thm}
\label{thm:naz-pol-sod05}
Under the assumptions above in dimension $n=2$
$$\Vol(\vphi>0) \geq \frac{C}{\log \beta}\ .$$
\end{thm}
Theorem~\ref{thm:naz-pol-sod05} was proved by using complex-analysis methods,
and we hope that estimate~(\ref{harmonic-vol}) on each positive component separately
may lead to a real analysis proof of Theorem~\ref{thm:naz-pol-sod05}.

We show by a series of examples on the $n$-dimensional round
sphere $\Sb^n$, that estimate~(\ref{local-courant-small}) is sharp
up to the $(\log\lambda)^{n-1}$ factor:
\begin{thm}
Consider the standard round sphere $\Sb^n$.
For every eigenvalue $\lambda$ and $R<1/\sqrt{\lambda}$
there exists an eigenfunction $\vphi_\lambda$ on $\Sb^n$,
a nodal domain $\Al_\lambda$ and a ball $B$ of radius $R$
such that
$$\frac{|\Al_\lambda\cap B|}{|B|} \leq
\frac{C_3(n)}{(\sqrt{\lambda})^{n-1}}\ ,$$
and $\Al_\lambda\cap \Bhf\neq\emptyset$.
\end{thm}
For $R \sim 1$,
we have
\begin{equation}
\label{frm-large-balls}
|\Omega_\lambda|/|B| \geq C_2(\sqrt{\lambda})^{-(2n-2)}
(\log\lambda)^{-(n-1)} \ .
\end{equation}
An example (see Section~\ref{sec:example-tn})
on a flat torus suggests that the power $2n-2$
in~(\ref{frm-large-balls}) could be improved to $n-1$. However, we
believe that the large ball behavior should depend on the dynamics
of the manifold, and it would be interesting to prove a result 
in this direction. 

We also consider in this paper $C^{\infty}$-surfaces. In this case, we use the
methods from~\cite{NazPolSod05} in order to reduce the analysis to
the case of flat metrics. Then, we can apply complex analysis. We
prove
\begin{thm}
 \label{thm:local-courant-dim-2}
   Let $(\Sigma, g)$ be a closed $C^{\infty}$-Riemannian surface.
     Then
     $$\frac{|\Omega_\lambda|}{|B|}\geq
     \frac{C_4}
{\lambda R' (\log\lambda)^{1/2}}\ ,$$ where $\lambda,
\vphi_\lambda, B, R, R', \Omega_\lambda$ are as in
Theorem~\ref{thm:local-courant-analytic}.
\end{thm}
\begin{remark}
Comparing the last theorem with
Theorem~\ref{thm:local-courant-analytic} in the case $n=2$, we see
that we gain a $(\log\lambda)^{1/2}$ factor. This is due to
complex analysis methods.
\end{remark}

\vspace{1Ex}
 One can interpret
Theorem~\ref{thm:local-courant-dim-2} as an estimate on the size
of the so called ``avoided crossings'' discussed in the physics
literature (\cite{smilansky}). Roughly speaking, two nodal lines
cannot approach each other too much for a long period of time.
Theorem~\ref{thm:local-courant-dim-2} shows that two nodal lines
cannot approach each other much less than a distance of
$1/\lambda$ (interestingly, no square root here) along a line of
length $\geq C/\sqrt{\lambda}$. On the other hand, we showed
in~\cite{Man-IRND} that two adjacent nodal lines cannot stay much
closer than $C/\sqrt{\lambda}$ at all times.


%
%
%
%
%

\subsection{Methods of Proof}
The main tool which we use in the proof of
Theorem~\ref{thm:local-courant-analytic}
is a generalization of Hadamard's 3-circles theorem due to Nadirashvili.
This lets us eliminate difficult Carleman type estimates
which were used in~\cite{DonFef90}.
In more details we exploit the following three properties
of eigenfunctions:

\bigskip
\noindent\textbf{``Reduction'' to Harmonic Functions.}
We follow the principle that on balls of small
radius with respect to the wavelength $1/\sqrt{\lambda}$ a
$\lambda$-eigenfunction is almost harmonic.
This principle was developed in \cite{DonFef88}, \cite{DonFef90}, \cite{Nad91}
and~\cite{NazPolSod05}.
After rescaling an eigenfunction $\vphi_\lambda$
in a ball of radius $\sim 1/\sqrt{\lambda}$ to the unit ball $B_1$,
one arrives at a solution $\vphi$ of a second order
self adjoint elliptic operator $L$ in the
unit ball $B_1\subseteq\Rb^n$,
where $L$ has coefficients bounded independently of $\lambda$.
$\vphi$ is close to a harmonic function in a sense to be clarified below.
Moreover, the growth (defined below) of $\vphi$ in the unit ball is bounded
in terms of $\lambda$ (\cite{DonFef88}).



\bigskip
\noindent\textbf{Rapid Growth in Narrow Domains.}
If a harmonic function $\vphi$ vanishes on the boundary of a domain
which is long and narrow $\vphi$ must grow exponentially fast
along the direction in which the domain is long.
A corresponding property is true for eigenfunctions on
any $C^\infty$-manifold.

This was extensively developed and investigated
by Landis (\cite{Lan63}) for a certain class of solutions
of second order elliptic equations. The version we found in~\cite{Lan63}
cannot be directly applied to eigenfunctions.
A version for eigenfunctions but
with slightly weaker estimates than in the present paper
 was proved in~(\cite{DonFef90}).

In Section~\ref{sec:properties}
we formulate a sharp version of this property
which can be applied for eigenfunctions. We prove it
in Section~\ref{sec:rapid-growth-narrow-domains}.
The proof combines the ideas from~\cite{Lan63} and~\cite{DonFef90}.
We replace some arguments from~\cite{DonFef90} by more elementary ones.

\bigskip
\noindent\textbf{Nadirashvili's-Hadamard's 3-Circles Theorem.}
This is a ``propagation of smallness'' principle:
Let $\vphi$ be a harmonic function in the unit ball,
and suppose $|\vphi|\leq 1$.
If a harmonic function $\vphi$ is small on
a subset $E\subseteq B_r$, where $r<1$ and $|E|/|B_r|$ is large
then $|\vphi|$ can be estimated from above in any 
concentric ball $B_R$ containing $B_r$.
When $E$ is a ball centered at $0$, this reduces to the classical
Hadamard's 3-Circles Theorem. Nadirashvili replaced the innermost
circle by an arbitrary measurable subset $E$.
When this principle is adapted to eigenfunctions,
we are restricted to consider real analytic metrics.
The sharp estimate in the generalized Hadamard Theorem is the main source
from which we get the improvement in Theorem~\ref{thm:local-courant-analytic}
relative to Theorem~\ref{thm:local-courant}.

%
%
%

\subsection{Organization of the paper}
In Section~\ref{sec:rescaling} we rescale the problem on balls of
small radius compared with the wavelength to a problem on the unit
ball in $\Rb^n$. Thus, we arrive to consider a problem on almost
harmonic functions whose growth is controlled. In
Section~\ref{sec:properties} we explain Rapid Growth in Narrow
Domains and Propagation of Smallness. In Section~\ref{sec:proof}
we prove Theorem~\ref{thm:local-courant-analytic}. In
Section~\ref{sec:rapid-growth-narrow-domains} we prove Rapid
Growth in Narrow Domains. In Section~\ref{sec:dimension2} we
consider the case of smooth surfaces. We give a new estimate for
harmonic functions in dimension two, and show how
Theorem~\ref{thm:local-courant-dim-2} follows.
 In Section~\ref{sec:examples} we give two examples.
The first one is a sequence of spherical harmonics which
demonstrate that Theorem~\ref{thm:local-courant-analytic} is sharp
up to a logarithmic factor for balls of radius smaller than the
wavelength. The second one is a sequence of eigenfunctions on the
standard flat torus concerning the sharp bound in
Theorem~\ref{thm:local-courant-analytic} for balls of radius $\sim
1$.


\bigskip
\noindent\textbf{Acknowledgements:} I would like to heartily thank
Kolya Nadirashvili for indicating to me that the Generalized
Hadamard Theorem may be fruitful. I owe many thanks to Leonid
Polterovich and Misha Sodin for their continuous encouragement and
helpful discussions. I am grateful to Alexander Eremenko for
removing the extra log factor in our estimate in dimension two.
Finally, I would like to thank Alexander Borichev and 
Dima Jakobson for fruitful discussions. 
This paper was written in the IHES, MPIM-Bonn while the author
was an EPDI fellow, and the CRM, Montreal. 
The support of the EPDI, IHES, MPIM-Bonn and the CRM 
is gratefully acknowledged.

\section{Passing to the Wavelength Scale}
\label{sec:rescaling}
Consider an eigenfunction $\vphi_\lambda$ on a ball $B\subseteq M$
of radius small compared to the wavelength. Namely,
consider $\vphi_\lambda$ as a function defined on a Euclidean ball of radius 
$\sqrt{\eps/\lambda}$ contained in $\Rb^n$ as a coordinate neighbourhood. Now, 
rescale $\vphi_\lambda$ to a function $\vphi$
on the unit ball $B_1\subseteq\Rb^n$. $\vphi$ satisfies
\begin{equation}
\label{solution}
      L\vphi=0\ ,
\end{equation}
where $L$ is a second order elliptic operator with $C^{\infty}$ coefficients.
Equation~(\ref{solution}) with real analytic coefficients will be denoted by~(\ref{solution}.RA).
$L$ is of the form
\begin{equation}
\label{L-definition}
   Lu := -\partial_i (a^{ij}\partial_j u) -\eps qu \ .
\end{equation}
$a^{ij}$ is symmetric and satisfies ellipticity bounds:
\begin{equation}
\label{ellipticity}
   \kappa_1 |\xi|^2\leq a^{ij}\xi_i\xi_j \leq \kappa_2 |\xi|^2.
\end{equation}

$a^{ij}, q$ are bounded:
\begin{equation}
\label{bounds}
  \|a^{ij}\|_{C^1(\overline{B_1})} \leq K,\ |q| \leq K\ ,
 \end{equation}
and we will assume $\eps<\eps_0$, and $\eps_0$ is small.

It is a fundamental fact due to Donnelly and Fefferman 
that the eigenvalue $\lambda$
controls the growth of the eigenfunction $\vphi_\lambda$ 
and hence also of $\vphi$:
For a Euclidean ball $B\subseteq B_1$, we define the $r$-growth exponent as:
\begin{eqnarray}
\label{def:beta}
\beta_r(\vphi; B)
:=\log\frac{\sup_{B}|\vphi|}{\sup_{rB}|\vphi|}
 &,&\beta_r(\vphi) :=\sup_{B\subseteq B_1} \beta_r(\vphi; B)\ .
\end{eqnarray}
Donnelly and Fefferman proved~(\cite{DonFef88})
that for the rescaled eigenfunction $\vphi$ we have
\begin{equation}
\label{DF-growth} \frac{\beta_r(\vphi)}{\log(1/r)}\leq
C\sqrt{\lambda} \ .
\end{equation}

\section{Rapid Growth and Propagation of Smallness}
\label{sec:properties}
In this section we give precise formulations of the properties
mentioned in the introduction. We use them in Section~\ref{sec:proof}.

\begin{itemize}
\item \textbf{Rapid Growth in Narrow Domains}:
This property tells that if  a solution $\vphi$ of~(\ref{solution})
has a deep and narrow positivity component $\Omega$,
then $\vphi$ grows rapidly in $\Omega$.
A first version of it
is given by
\begin{thm}
\label{thm:rapid-growth-1st-version}
Let $\vphi$ satisfy~(\ref{solution}).
Let $\Omega$ be a connected component
of \mbox{$\{\vphi>0\}$} which intersects $B_{1/2}$.
Let $\eta>0$ be small enough. If
$$\frac{|\Omega|}{|B_1|} \leq \eta^{n-1} \ ,$$
then
   $$\frac{\sup_{\Omega}{\vphi}}{\sup_{\Omega\cap B_{1/2}} \vphi} \geq
     \me^{C_1/\eta}\ .$$
\end{thm}
We remark that $\eta$ can be considered as a bound on the cross sections
of $\Omega$.
We will use an iterated version of the above property:
\begin{thm}
\label{thm:rapid-growth}
Let $\vphi$ satisfy~(\ref{solution}).
Let $0<r_0\leq 1/2$. Let $\Omega$ be a connected component of
\mbox{$\{\vphi>0\}$} which intersects $B_{r_0}$.
Let $\eta>0$ be small enough. If $|\Omega\cap B_r|/|B_r| \leq \eta^{n-1}$
for all $r_0<r < 1$, then
$$\frac{\sup_{\Omega} \vphi}
{\sup_{\Omega\cap B_{r_0}} \vphi} \geq
\left(\frac{1}{r_0}\right)^{C_1/\eta}\ . $$
\end{thm}
%

\item \textbf{Nadirashvili's-Hadamard's 3-Circles Theorem:}
\begin{thm}[\cite{Nad76}]
\label{thm:nad-had}
Let $\vphi$ satisfy~(\ref{solution}.RA),
Let $E\subseteq B_R$, where $R<1$.
\begin{equation*}
 \mbox{If } \frac{\sup_E |\vphi|}{\sup_{B_1} |\vphi|}
\leq \left(\frac{|E|}{|B_1|}\right)^{\gamma/n},\  \mbox{then }
\frac{\sup_{B_R} |\vphi|}{\sup_{B_1}|\vphi|}
\leq (c_0 R)^{c_1\gamma},
\end{equation*}
whenever $\gamma>\gamma_0$, and where
$\gamma_0, c_0, c_1$ depend on $\kappa_1, \kappa_2, K, \eps_0, n$.
\end{thm}
When $L$ is the Euclidean Laplacian, $n=2$ and $E$ is a ball centered
at $0$ we get the classical Hadamard Theorem which says that
the maximal function $\max_{B_r} |\vphi|$ is a logarithmic convex function of
$\log r$. When $E=B_R$ and $n>2$,
this theorem was proved by Gerasimov in~\cite{Ger66}.

\begin{observation} Theorem~\ref{thm:nad-had} is equivalent to
\begin{equation}
\label{propagation}
  \sup_{B_1} |\vphi| \leq \sup_E |\vphi| \cdot
  \left(\frac{|B_1|}{|E|}\right)^{C\beta_R(\vphi; B_1)/\log(1/R)}\ ,
\end{equation}
for all $R$ small enough and where $E\subseteq B_R$.
One should compare~(\ref{propagation}) with inequality~(D) 
in~\cite[p. 636]{DonFef90}.
\end{observation}

\begin{remark}
The proof in~\cite{Nad76} is for harmonic functions. When one goes
through the proof, one sees that the only point where the
harmonicity of $\vphi$ is used is an interior elliptic regularity
estimate which is true for solutions of any second order elliptic
operator with real analytic coefficients (See~\cite[Theorem
7.5.1]{Hor64}).
\end{remark}

\end{itemize}
\section{Proof of Theorem~\ref{thm:local-courant-analytic}}
\label{sec:proof}
From Section~\ref{sec:rescaling} it follows that on balls of
radius $R<\sqrt{\eps_0/\lambda}$
Theorem~\ref{thm:local-courant-analytic} is equivalent to
Theorem~\ref{thm:local-courant-analytic}' below. In section~\ref{sec:bigballs}
we complete the argument for larger balls.

Recall the definition~(\ref{def:beta}) of $\beta_r$. We set
\begin{equation*}
\beta_r'(\vphi; B) := \max\{\beta_r(\vphi; B), 3\},\quad
\beta_r'(\vphi):=\max(\beta_r(\vphi), 3)\ .
\end{equation*}
\begin{thm1.3p}
\label{thm:rescaled}
Let $\vphi$ satisfy~(\ref{solution}.RA).
Let $\Omega$ be a
 connected component of $\{\vphi\neq 0\}$. Suppose $\Omega\cap B_{1/2}\neq\emptyset$.
Then, $$\frac{|\Omega|}{|B_1|}
\geq \frac{C_1}{(\beta_{\rho_0}'(\vphi)\log\beta_{\rho_0}'(\vphi))^{n-1}}\ ,$$
where $0<\rho_0<1$ depends only on $K, \kappa_1, \kappa_2, \eps_0$ and $n$.
\end{thm1.3p}

Theorem~\ref{thm:local-courant-analytic}' immediately follows from
\begin{thm1.3pp}
Let $\vphi$ satisfy~(\ref{solution}.RA).
Suppose $\vphi(0)>0$. Let $\Omega$ be the connected component of
$\{\vphi>0\}$ which contains $0$. Then
$$\frac{|\Omega|}{|B_1|}
\geq \left(\frac{C_1\eps}
{\beta_{\rho_0}'(\vphi; B_1)}\right)^{(n-1)/(1-\eps)}$$
for all $0<\eps<1/2$, and where $0<\rho_0<1$ depends only on
$K, \kappa_1, \kappa_2, \eps_0, n$.
\end{thm1.3pp}
\begin{proof}[Theorem~\ref{thm:local-courant-analytic}'' 
implies Theorem~\ref{thm:local-courant-analytic}']
One may use an affine transformation in order
to move the center of the ball to $\Omega\cap B_{1/2}$.
Also, one can check that
$$\sup_{0<\eps<1/2} (\eps/A)^{1/(1-\eps)} \geq C/A\log A\ .$$
Indeed,
   with $\eps= 1/\log A$, one has
   $$\left(\frac{\eps}{A}\right)^{1/(1-\eps)} \geq
  \frac{\eps}{A}\left(\frac{\eps}{A}\right)^{2\eps} =
  \frac{\me^{2\eps\log\eps - 2}}{A\log A} \geq \frac{\me^{-3}}{A\log A}\ .$$
\end{proof}

We now give the heart of the argument which proves 
Theorem~\ref{thm:local-courant-analytic}''. It is close to the argument
in~\cite{DonFef90}. The main new tool we introduce in the argument
is the propagation of smallness principle in the form of
Nadirashvili's-Hadamard's 3-circles Theorem.
\begin{proof}[Idea of Proof of Theorem~\ref{thm:local-courant-analytic}'']
If $|\vphi|\leq 1$ in the ball $B_1$
and $|\Omega|$ is very small, then by Rapid Growth in Narrow Domains
$\vphi$ should be very small on $E_{1/2}:=|\Omega\cap B_{1/2}|$.
Then by the propagation of smallness
from $E_{1/2}$ to $B_{1/2}$, we know that $|\vphi|$ must be small also
on the ball $B_{1/2}$, but if it is too small we will get
that the growth of $\vphi$ on $B_1$ is bigger than $\beta_{1/2}(\vphi; B_1)$.
The argument above has to be modified a little bit,
since we can apply propagation of smallness only if we know that
the relative size of $E_{1/2}$ in $B_{1/2}$ is large.
All one has to do is to replace $E_{1/2}$ by $E_r$ for
a suitable $r<1/2$.
\end{proof}

We now give the complete proof.
\begin{proof}[Proof of Theorem~\ref{thm:local-courant-analytic}'']
We may assume
\begin{equation}
\label{assumptioneta}
  |\Omega|/|B_1| < (\eps/\gamma_1)^{(n-1)/(1-\eps)},
\end{equation}
where $\gamma_1$ is a large enough constant.
Otherwise, the theorem becomes trivial.
%
Let $c_0>1$ be as in Theorem~\ref{thm:nad-had}, and let
$\rho_0 = 1/(2c_0)$.
Define
\begin{equation}
\label{r0-definition}
r_0 := \sup\left\{r:\, \frac{|\Omega\cap B_r|}{|B_r|}
\geq \frac{1}{\rho_0^n}\left(\frac{|\Omega|}{|B_1|}\right)^{1-\eps}\right\}\ .
\end{equation}
Let $E:=\Omega\cap B_{r_0}$.
\begin{claim}
$$r_0 <\rho_0 (|E|/|B_1|)^{\eps/n}\ .$$
\end{claim}
\begin{proof}[Proof of Claim]
By definition,
$$(r_0/\rho_0)^n\leq(|\Omega|/|B_1|)^{\eps}
(|E|/|\Omega|)\leq
(|\Omega|/|B_1|)^{\eps}(|E|/|\Omega|)^{\eps} .$$
\end{proof}
In particular,
\begin{equation}
\label{r0-romegacapbr0} r_0 \leq \min\{\rho_0,
(|E|/|B_1|)^{\eps/n}\}\ .
\end{equation}
Theorem~\ref{thm:rapid-growth} and
Inequality~(\ref{r0-romegacapbr0}) give together
\begin{multline}
\label{small-in-omegar0}
  \log\frac{\sup_{E} |\vphi|}{\sup_{B_1} |\vphi|}\leq \log\frac{\sup_{E} |\vphi|}{\sup_{\Omega} |\vphi|}
\leq
C_1(\log r_0)\rho_0^{n/(n-1)}/(|\Omega|/|B_1|)^{(1-\eps)/(n-1)}\leq \\
 C_2\eps(\log(|E|/|B_1|)^{1/n})/(|\Omega|/|B_1|)^{(1-\eps)/(n-1)} \ .
\end{multline}
Due to assumption~(\ref{assumptioneta}) and to $E\subseteq B_{\rho_0}$
we can apply Theorem~\ref{thm:nad-had}:
$$-\beta_{\rho_0}(\vphi; B_1)
=\log\frac{\sup_{B_{\rho_0}} |\vphi|}{\sup_{B_1} |\vphi|}
\leq {C_3\eps (\log c_0 \rho_0)/(|\Omega|/|B_1|)^{(1-\eps)/(n-1)}}\ .$$
The last inequality can be rewritten as
 $$\frac{|\Omega|}{|B_1|} \geq
\left(\frac{\eps C_4}{\beta_{\rho_0}(\vphi; B_1)}\right)^{(n-1)/(1-\eps)} \ .$$
\end{proof}

\subsection{Handling Large Balls}
\label{sec:bigballs}
So far, we proved Theorem~\ref{thm:local-courant-analytic} for small balls.
We now treat the case of balls $B$ whose radius
radius $R> \sqrt{\eps_0/\lambda}$. There are
two cases: \begin{itemize}
\item[1)] $\Omega_\lambda=\Al_\lambda\subseteq B$. By the Faber-Krahn
Inequality, we know that
  $|\Omega_\lambda| \geq C_2/(\sqrt{\lambda})^n$.
Hence,
$$\frac{|\Omega_\lambda|}{|B|} \geq
\frac{C_3}{(R\sqrt{\lambda})^n}\ ,$$ which gives the estimate in
Theorem~\ref{thm:local-courant-analytic}.

\item[2)] $\Omega_\lambda$ touches $\partial B$. We
decompose $B\setminus \Bhf$ into spherical layers, each of width
$\sqrt{\eps_0/\lambda}$. In each spherical layer we can find a
ball $B'$ of radius $(\sqrt{\eps_0/\lambda})/2$ such that
$\Omega_\lambda$ cuts $\Bhf'$. The number of such layers is
$\sim R\sqrt{\lambda/\eps_0}$. By Theorem~\ref{thm:local-courant-analytic}
for small balls the total
volume of $\Omega_\lambda$ is
$$|\Omega_\lambda|\geq \sum_{B'} |\Omega_\lambda\cap B'|\geq
C_5
R\sqrt{\lambda/\eps_0}|B'|/(\sqrt{\lambda}\,\log\lambda)^{n-1}\
.$$

The last inequality gives
$$\frac{|\Omega_\lambda|}{|B|}\geq
\frac{C_6}{(R\sqrt{\lambda})^{n-1}(\sqrt{\lambda}\,\log\lambda)^{n-1}}\ .$$
\end{itemize}
This completes the proof of Theorem~\ref{thm:local-courant-analytic}.
\section{Rapid Growth in Narrow Domains}
\label{sec:rapid-growth-narrow-domains}
In this section we prove Theorems~\ref{thm:rapid-growth-1st-version}
and~\ref{thm:rapid-growth}.
They follow from the classical growth Lemma:
Let $\vphi$ satisfy~(\ref{solution}).
Let $B^y_R=B(y, R)\subseteq B_1$.
Suppose $\vphi(y)>0$, and let
$\Omega_y$ be the connected component of $\{\vphi>0\}\cap B^y_R$
which contains $y$.
The Growth Lemma is
\begin{lem}[{\cite[Lemma~3.1]{Lan63}, \cite[p.~651]{DonFef90}}]
\label{lem:growth-lemma}
For all $A>1$ there exists $\gamma(A)>0$ such that if
$$\frac{|\Omega_y|}{|B^y_R|} \leq \gamma(A)\ ,$$
then
$$\frac{\sup_{\Omega_y} |\vphi|}
{\sup_{\Omega_y \cap B^y_{R/2}} |\vphi|} \geq A \ . $$
In particular, $\gamma$ does not depend on $R$, neither on $y$.
\end{lem}
We give a proof of this lemma in Section~\ref{sec:growth-lemma}.
As a corollary we obtain
Theorem~\ref{thm:rapid-growth-1st-version}:

\begin{proof}[Proof of Theorem~\ref{thm:rapid-growth-1st-version}]
We notice that $\Omega$ touches $\partial B_1$. Otherwise, since $\eps<\eps_0$
the maximum principle (\cite[Cor. 3.8]{GilTru83}) tells us that $\vphi$ is
identically $0$ in $\Omega$, which contradicts the definition of $\Omega$.

We decompose $B_1\setminus B_{1/2}$ into $N$ equally distanced
spherical layers, where $N$ will be chosen below.
Let $t_k = (1/2+k/(2N))$, $k=0\ldots N$.
Let $A_0 = B_{1/2}$, and  $A_k = B_{t_k}\setminus B_{t_{k-1}}$
for $k=1\ldots N$.
Set  $\wtld{A_0} = A_0\cup A_1$, $\wtld{A_N} = A_{N-1}\cup A_N$,
and
  $$\wtld{A_k} = A_{k-1}\cup A_k\cup A_{k+1}$$
for $1\leq k\leq N-1$.
There exist $\geq N/2$ values of $1 \leq k\leq N$ for which
$$|\Omega\cap \wtld{A_k}|\leq 6|\Omega|/ N\ .$$
Let $l$ be any one of these values.
Let $y\in \Omega\cap A_l$, and let $R=1/(2N)$.
Consider the ball $B^y_R=B(y, R)$.
$B^y_R\subseteq \wtld{A_l}$, and we check that
\begin{equation}
\label{growth-lemma-start}
\frac{|\Omega\cap B^y_R|}{|B^y_R|} \leq
\frac{|\Omega\cap \wtld{A_l}|}{|B_1|(1/(2N))^{n}}
\leq \frac{6|\Omega|(2N)^n}{N|B_1|}\leq 12(2\eta N)^{n-1} \ .
\end{equation}

Set $A=10\me/9$ and take
$N =\lfloor((\gamma(A)/12)^{1/(n-1)}/(2\eta))\rfloor$.
Inequality~(\ref{growth-lemma-start}) and Lemma~\ref{lem:growth-lemma}
applied in $B^y_R$ show that
$$\sup_{\Omega\cap \wtld{A_l}} \vphi \geq
\sup_{\Omega\cap B^y_{R}} \vphi \geq A\vphi(y)\ .$$
Since this is true for all $y\in \Omega\cap A_l$, we get
\begin{equation}
\label{almost-recursion}
  \sup_{\Omega\cap B_{t_{l+1}}}\vphi 
\geq \sup_{\Omega\cap \wtld{A_l}} \vphi \geq
A\sup_{\Omega\cap A_l}\vphi\ .
\end{equation}
Now we apply the following maximum principle:
\begin{thm}[{\cite[Corollary~3.8]{GilTru83}}]
Let $\vphi$ satisfy~(\ref{solution}). Then
$$\sup_{\Omega\cap A_k} \vphi \geq 0.9\sup_{\Omega\cap B_{t_k}}\vphi\ ,$$
whenever $\eps<\eps_0$.
\end{thm}
Hence, from~(\ref{almost-recursion}) we obtain
 $$ \sup_{\Omega\cap B_{t_{l+1}}}\vphi
    \geq 0.9 A\sup_{\Omega\cap B_{t_l}}\vphi =
    \me \sup_{\Omega\cap B_{t_l}}\vphi\ .$$
And since this is true for $\geq N/2$ values of $k$ we finally have
$$ \frac{\sup_{\Omega} \vphi}{\sup_{\Omega\cap B_{1/2}}\vphi} \geq
\me^{N/2}
\geq \me^{C_1/\eta}\ .$$
\end{proof} 

An iteration of Theorem~\ref{thm:rapid-growth-1st-version} gives
Theorem~\ref{thm:rapid-growth}:
\begin{proof}[Proof of Theorem~\ref{thm:rapid-growth}]
Let $N$ be a positive integer for which $(1/2)^{N+1}<r_0 \leq (1/2)^N$.
$N=\lfloor\log(1/r_0)/\log 2\rfloor$.
Set $t_k = (1/2)^k$. It follows by scaling from
Theorem~\ref{thm:rapid-growth-1st-version} that
$$\frac{\sup_{\Omega\cap B_{t_k}} |\vphi|}
{\sup_{\Omega\cap B_{t_{k+1}}} |\vphi|} \geq
\me^{C_1/\eta}\ ,$$
for all $0\leq k \leq N-1$.
The point is that the bounds~(\ref{ellipticity})-(\ref{bounds})
on the operator $L$ remain true
after rescaling.
Hence,
\begin{equation*}
  \frac{\sup_{\Omega} |\vphi|}{\sup_{\Omega\cap B_{r_0}} |\vphi|}
\geq
\frac{\sup_{\Omega} |\vphi|}{\sup_{\Omega\cap B_{t_N}} |\vphi|} \geq
\me^{C_1 N/\eta} 
\geq \me^{C_1\log(1/r_0)/(2\eta\log 2)} =
(1/r_0)^{C_2/\eta}\ .
\end{equation*}
\end{proof}

\subsection{Proof of the Growth Lemma}
\label{sec:growth-lemma}
Its proof is based on ideas
from the proof in~\cite{DonFef90},
where we replaced several arguments by more elementary ones.

\begin{proof}[Proof of Lemma~\ref{lem:growth-lemma}]
Let $g(t)$ be a smooth function defined on $\Rb$
with the following properties
\begin{itemize}
\item $g(t) = 0$ for $t\leq 1$,
\item $g(t) = t-2$ for $t\geq 3$,
\item $g''(t) \geq 0$.
\end{itemize}
Let $\delta>0$ be small and
let $g_\delta(t) = \delta g(t/\delta)$.
Let $\vphi_\delta = (g_\delta\circ \vphi) \cdot \chi_{\Omega_y}$,
where $\chi_{\Omega_y}$ is the characteristic function of $\Omega_y$.
$\vphi_\delta$ is a smooth function in $B^y_R$ with 
compact support in $\Omega_y$.
We notice that $0\leq\vphi\chi_{\Omega_y}-\vphi_\delta\leq 2\delta$.
In particular, $\vphi_\delta\to\vphi\chi_{\Omega_y}$
uniformly as $\delta\to 0$.
We now calculate $L\vphi_\delta$:
\begin{multline}
\label{almost-subharmonic}
  L\vphi_\delta = \chi_{\Omega_y}\eps q \vphi\cdot(g_\delta'\circ\vphi)
-\eps q \vphi_\delta-\chi_{\Omega_y}a^{ij}\partial_i\vphi\partial_j
\vphi\cdot(g_\delta''\circ\vphi) \\ \leq
  \chi_{\Omega_y}\eps q\vphi\cdot(g_\delta'\circ\vphi)
-\eps q \vphi_\delta\ .
\end{multline}
The last inequality is true due
to the fact that $a^{ij}$ is positive definite and
$g$ is convex. Let us denote the last expression in~(\ref{almost-subharmonic})
by $f$.
\begin{lem}
\label{lem:ln-norm}
  $$\|f\|_{L^n(B^y_R)} \leq 3\eps_0 \delta K |\Omega_y|^{1/n}$$
\end{lem}
We postpone the proof of this lemma to the end of this section.
Recall the
following local maximum principle:
\begin{thm}[{\cite[Theorem 9.20]{GilTru83}}]
Suppose $Lu\leq f$ in $B^y_R\subseteq B_1$. Then,
$$\sup_{B^y_{R/2}} u
 \leq \frac{C_1}{|B^y_R|}\int_{B^y_R} u^+ \dx + C_2\|f\|_{L^n(B^y_R)}\ ,$$
where $C_1, C_2$ depend only on $\kappa_1, \kappa_2$ and $K$.
\end{thm}
Applying the local maximum principle to $\vphi_\delta$ gives
$$\sup_{B^y_{R/2}} \vphi_\delta \leq \frac{C_1}{|B^y_R|}\int_{B^y_R}
\vphi_\delta\dx +
3C_2\delta\eps_0 K |\Omega_y|^{1/n} \ .$$
Letting $\delta\to0$ we obtain that
$$\sup_{\Omega_y\cap B^y_{R/2}}\vphi\leq
\frac{C_1}{|B^y_R|}\int_{\Omega_y} \vphi \dx
\leq \frac{C_1|\Omega_y|}{|B^y_R|} \sup_{\Omega_y} \vphi\ . $$
Thus, we may take $\gamma(A) = 1/(C_1 A)$.
\end{proof}
To complete the proof of the Growth Lemma~\ref{lem:growth-lemma}
it remains to prove Lemma~\ref{lem:ln-norm}.
\begin{proof}[Proof of Lemma~\ref{lem:ln-norm}]
We notice that $t-2\delta\leq g_\delta(t)\leq t$
and $0\leq g_\delta'(t) \leq 1$.
Hence,
$f\leq \eps q\vphi\chi_{\Omega_y}
- \eps q (\vphi-2\delta)\chi_{\Omega_y} = 2 \eps \delta q\chi_{\Omega_y}$.
When $\vphi \geq 3\delta$, $f=2\eps\delta q\chi_{\Omega_y}$,
and when $\vphi\leq 3\delta$, $f\geq -\eps q \vphi \chi_{\Omega_y}
\geq -3\eps q \delta\chi_{\Omega_y}$.
We have shown that $|f| \leq 3\eps_0 \delta q\chi_{\Omega_y}$.
Integration gives
$$\|f\|_{L^n(B^y_R)} \leq 3 \eps_0 \delta K |\Omega_y|^{1/n}.$$
\end{proof}

\section{Dimension Two - $C^{\infty}$ case}
\label{sec:dimension2}
In dimension two we can use the techniques in~\cite{NazPolSod05}
in order to reduce the analysis to the case of harmonic functions.

\subsection{An estimate for harmonic functions}
Let $\vphi$ be a harmonic function in the unit ball $B_1\subseteq\Rb^2$.
Let $\Omega$ be a connected component of $\{\vphi\neq 0\}$.
Suppose $\Omega\cap B_{1/2}\neq\emptyset$.
Theorem~\ref{thm:local-courant-analytic}' shows
$$|\Omega|/|B_1|>C/\beta_{\rho_0}'(\vphi)\log\beta_{\rho_0}'(\vphi)$$
for some $0<\rho_0<1$.
We will show that in fact,
\begin{thm}
\label{thm:dim2-harmonic}
$$|\Omega|/|B_1| \geq C/\beta_{1/2}'(\vphi; B_1)\ .$$
\end{thm}

In order to see this we define
$h=\vphi$ in $\Omega$ and $h=0$ elsewhere in $B_1$. $h$
is a subharmonic function. The key is the
following proposition:
\begin{thm}[Eremenko]
\label{thm:eremenko}
$$\beta_{3/4}(h; B_1) \leq C_1\beta_{1/2}(\vphi; B_1)+ C_2\ .$$
\end{thm}
\begin{remark}
Nadirashvili's-Hadamard's Theorem is equivalent to a related inequality with
 $C_1=C_1(\Omega)$. Namely,
\begin{equation}
\label{nad-had-equiv}
\frac{\beta_r(h; B_1)}{\log(1/(|\Omega\cap B_r|/|B_1|)^{1/n})}
\leq C_1 \frac{\beta_r(\vphi; B_1)}{\log1/(C_2 r)}\ .
\end{equation}
The important point in Theorem~\ref{thm:eremenko} is that by
allowing a slightly bigger $r$ in the nominator of the LHS of~(\ref{nad-had-equiv}) 
we obtain a constant
$C_1$ which is independent of $\Omega$.
\end{remark}
\begin{proof}[Proof of Theorem~\ref{thm:eremenko}]
The proof is based on the harmonic majorant principle and on
the Beurling-Nevannlina projection theorem.
For a function $u$ defined in $B_1$, 
let $M_u(r):=\max_{|z|\leq r} |u|$. We normalize $\vphi$ 
so that $M_{\vphi}(1)=1$.
It is well known (\cite{hor-convexity}) 
that $M_h(r)$ is a convex function of $\log r$,
that is, $t M_h'(t)$ is a monotonically increasing function of $t$.
Thus, we have for all $1/2\leq t\leq2/3$

\begin{multline*}
  \me^{-\beta_{3/4}(h; B_1)}\geq M_h(3/4)-M_h(t) = \int_{t}^{3/4} 
\frac{rM_h'(r)}{r}\,\dif r \geq \\
tM_h'(t)4(3/4-t)/3  \geq tM_h'(t) / 9 \geq M_h'(t)/18 \ .
\end{multline*}
Hence, $M_h'(t) < 20 \me^{-\beta_{3/4}(h; B_1)}$ for all $1/2\leq t\leq 2/3$.

Now, let $z_t$ be a point where $h(z_t)=M_h(t)$. Then,
$$0\leq \frac{\partial h}{\partial r}(z_t)\leq M_h'(t),\quad
\frac{\partial h}{\partial\theta}(z_t)=0\ .$$

Thus, we get that on $z_t$
\begin{equation*}
|\nabla \vphi(z_t)|\leq 20\me^{-\beta_{3/4}(h; B_1)}\ .
\end{equation*}
The function $v:= \log|\nabla \vphi|$ is subharmonic in $B_1$,
and on $z_t$, 
\begin{equation}
\label{v-zt}
v(z_t)\leq -\beta_{3/4}(h; B_1) +C_1\ .
\end{equation}
Since on $|z|\leq 1$, $|\vphi|\leq 1$,
on the circle $|z| = 2/3$, $|\nabla\vphi| \leq C_2$. So,
\begin{equation}
\label{v-23}
v|_{|z|=2/3} \leq C_3\ .
\end{equation}
Now we apply the harmonic majorant principle:
Let $\omega(z, \gamma, B_{2/3})$ be the harmonic measure of
$\gamma =\{z_t : 1/2\leq t\leq 2/3\}$ with respect
to $B_{2/3}$.
$\omega$ is a harmonic function on $B_{2/3}\setminus\gamma$, 
which tends to $1$ on $z_t$
and to $0$ on $|z|=2/3$.
Thus, in light of~(\ref{v-zt}) and~(\ref{v-23}), we get
\begin{equation}
\label{majorant}
v\leq (-\beta_{3/4}(h; B_1)+C_1) \omega + C_3 \leq 
-\beta_{3/4}(h; B_1)\omega +C_4\ .
\end{equation}
Also, since $M_\vphi(1/2)\geq \me^{-\beta_{1/2}(\vphi; B_1)}$,
there exists a point $z_0$, $|z_0|\leq 1/2$, where 
\begin{equation}
\label{gradient}
v(z_0) \geq
-\beta_{1/2}(\vphi; B_1) + C_5\ .
\end{equation}
By the Beurling-Nevannlina Projection Theorem (\cite{ahlfors-conf-inv})
\begin{equation}
\label{beurling-nevannlina}
\omega(z_0, \gamma, D) \geq \omega(-|z_0|; [1/2, 2/3], B_{2/3}) = C_6 \ .
\end{equation}

Combining (\ref{beurling-nevannlina}) with~(\ref{majorant}) 
and~(\ref{gradient}),
we obtain

$$-\beta_{1/2}(\vphi; B_1)+ C_5\leq v(z_0)
\leq -C_6 \beta_{3/4}(h; B_1) + C_4\ .$$
\end{proof}
\begin{proof}[Proof of Theorem~\ref{thm:dim2-harmonic}]
By Theorem~\ref{thm:rapid-growth-1st-version}
$$\beta_{3/4}(h; B_1) \geq C/(|\Omega|/|B_1|)\ .$$
On the other hand, Theorem~\ref{thm:eremenko} tells us
$$\beta_{3/4}(h; B_1) \leq C_1\beta_{1/2}(\vphi; B_1)+C_2\ .$$
Combining these two inequalities we get
$$|\Omega|/|B_1| \geq C_1/(\beta_{1/2}(\vphi; B_1) + C_2)
\geq C_3/\beta_{1/2}'(\vphi; B_1)\ .$$
\end{proof}

\subsection{Proof of Theorem~\ref{thm:local-courant-dim-2}}
The technique developed in~\cite{NazPolSod05}
shows that one can use a quasiconformal mapping in order
to pass to estimates for harmonic functions.

From Section~\ref{sec:rescaling} and by using conformal coordinates
we see that
Theorem~\ref{thm:local-courant-dim-2} is equivalent
for small balls to

\begin{thm}
\label{thm:quasiconformal-transfer}
Let $\Delta\vphi - \eps q \vphi = 0$ in $B_1\subseteq\CC$.
Suppose $\vphi(0)>0$ and let $\Omega\subseteq B_1$ be the connected
component of $\{\vphi>0\}$ which contains $0$. Then,
$$\frac{|\Omega|}{|B_1|} \geq
\frac{C_3}{\beta_{1/2}'(\vphi)
(\log\beta_{1/2}'(\vphi))^{1/2}} \ .$$
\end{thm}
\begin{proof}
The theorem follows from Theorem~\ref{thm:dim2-harmonic}
in exactly the same way which is explained in~\cite{NazPolSod05}.
So, we omit the proof.
\end{proof}
Then, for large balls we proceed in the same way as in 
Section~\ref{sec:bigballs}.
\section{Examples}
\label{sec:examples}
\subsection{An Example on $\Sb^n$}
\label{sec:example-sn}
In this section we show that Theorem~\ref{thm:local-courant-analytic}
for balls of radius $R<1/\sqrt{\lambda}$
is sharp up to the $(\log\lambda)^{n-1}$ factor.
The example we give will be a sequence of
spherical harmonics on the standard sphere $\Sb^n$.
Let us denote by $\Hl^n_k$ the
space of spherical harmonics on $\Sb^n$ of degree $k$.
\begin{prop}
\label{prop:max-nodal-domains}
There exists a sequence
 $(Y_k^n)_{k\geq 1} \in \Hl^n_k$
with the following properties:
\begin{enumerate}
\item The number of nodal domains of $Y_k^n$ is
$\geq c_{1,n} k^n$.
\item \label{prop:2} There exist $\geq c_{2,n} k^{n-1}$ nodal
domains of $Y_k^n$ which have the north pole on their boundary.
\end{enumerate}
\end{prop}

\begin{cor}
For every eigenvalue $\lambda$ and $r<1/\sqrt{\lambda}$
there exists an eigenfunction $\vphi_\lambda$,
a nodal domain $\Al_\lambda$ and a ball $B$ of radius $r$
such that
$$\frac{\Vol(\Al_\lambda\cap B)}{\Vol(B)} \leq
\frac{C_3(n)}{(\sqrt{\lambda})^{n-1}}\ ,$$
and $\Al_\lambda\cap \Bhf\neq\emptyset$.
\end{cor}
\begin{proof}
$\lambda = k(k+n-1)$ for some integer $k\geq 0$.
Let $Y^n_k$ be as in Proposition~\ref{prop:max-nodal-domains}.
Let $B$ be a ball of radius $r<1/k$ centered at the north pole.
By Proposition~\ref{prop:max-nodal-domains}
there exists a nodal domain $\Al_\lambda$
for which
$$\frac{\Vol(\Al_\lambda\cap B)}{\Vol(B)} \leq
\frac{C_4(n)}{k^{n-1}}\ .$$
The result follows since $\lambda\sim k^2$.
%
%
\end{proof}

We now prove Proposition~\ref{prop:max-nodal-domains}.
First, we introduce spherical coordinates and we
review elementary facts about spherical harmonics.
\begin{lemma}
A point on the sphere $\Sb^n$ is parametrized by
$(\theta_1, \ldots, \theta_{n-1}, \vphi)$,
where $0<\theta_l<\pi$, $0\leq \vphi\leq 2\pi$,
and
$$
\begin{array}{lcl}
x_1 &=& \cos \theta_1\ ,\\
\vdots & & \\
x_{n-1} &=& \sin\theta_1\ldots\sin\theta_{n-2}\cos\theta_{n-1} \ ,\\
x_{n} &=& \sin\theta_1\ldots\sin\theta_{n-1}\cos\vphi\ , \\
x_{n+1} &=& \sin\theta_1\ldots\sin\theta_{n-1}\sin\vphi\ .
\end{array}
$$
\end{lemma}

We recall the definition of the zonal spherical harmonics and
Legendre Polynomials. Details can be found in chapter~$3$ of~\cite{groemer}.
Consider the natural action of the orthogonal group
$O(n+1)$ on $\Sb^n$. It induces a representation of $O(n+1)$ on $\Hl_k^n$.
The \emph{zonal} spherical harmonic $Z^n_{k, p}$
of degree $k$ with pole $p\in \Sb^n$
is defined as the unique spherical harmonic in $\Hl_k^n$, which
is fixed by the stabilizer of the point $p$ in $O(n+1)$, and
admits the value~$1$ at~$p$.
The Legendre polynomial $P^{n+1}_k(t)$
is defined to be the polynomial on $[-1, 1]$, for which
$$Z^{n}_{k, p_0}(\theta_1, \ldots, \theta_{n-1}, \phi)
= P^{n+1}_k(\cos\theta_1)\ ,$$
where $p_0$ is the north pole.
It is easy to see that for any $p\in\Sb^n$
$$Z^n_{k, p}(x) = P^{n+1}_k(\langle p, x \rangle)\ .$$

\begin{lemma}[{\cite[Proposition 3.3.7]{groemer}}]
\label{lem:explicit-legendre}
$P^{n}_k$ is given by
   $$P^{n}_k(t)=\alpha_n (-1)^k (1-t^2)^{-(n-3)/2}
\frac{\partial^k}{\partial t^k} (1-t^2)^{k+\frac{n-3}{2}}\ ,$$
where $\alpha_n$ are some constants which depend on $n$.
\end{lemma}

We define also the associated Legendre functions:
 $$E^{n}_{k, j}(t) = (1-t^2)^{j/2}(\partial_t^j P^{n}_k)(t)\ .$$

The next lemma is an inductive construction of spherical harmonics:
\begin{lemma}[{\cite[Lemma~3.5.3]{groemer}}]
\label{lem:standard-basis}
  Given $G\in\Hl^{n-1}_j$, let
  $$H(\theta_1, \ldots\theta_{n-1},\vphi):=
E^{n+1}_{k, j}(\cos\theta_1) G(\theta_2, \ldots \theta_{n-1}, \vphi)\ . $$
Then, $H\in\Hl^n_k$.
\end{lemma}

\begin{proof}[Proof of Proposition~\ref{prop:max-nodal-domains}]
We prove it by induction on $n$.
For $n=1$, we take $Y_k^1(\vphi) = \sin k\vphi$.
Suppose the result is true for $n-1$.
Set
$$H^n_{k, j}(\theta_1,\ldots, \theta_{n-1},\vphi)
:= E^{n+1}_{k, j }(\cos\theta_1)
Y^{n-1}_{j}(\theta_2,\ldots, \theta_{n-1},\vphi)\ .$$

By Lemma~\ref{lem:standard-basis} $H^n_{k, j} \in \Hl^n_k$.
From Lemma~\ref{lem:explicit-legendre} one can see that
$E^{n+1}_{k, j}$ has exactly $k-j$ distinct zeroes in the interval $(-1, 1)$
it follows that
the number of nodal domains of $H^k_n$ is $\geq c_{1,n-1} (k-j) j^{n-1}$,
of which $c_{1,n-1}j^{n-1}$ touch the north pole.
We define
$Y^{n}_k := H^n_{k, \lfloor k/2\rfloor}$.
\end{proof}

%


\subsection{An Example on $\Tb^n$}
\label{sec:example-tn}
\begin{thm}
Consider the standard flat torus $\Tb^n$.
For every eigenvalue $\lambda$
there exists an eigenfunction $\vphi_\lambda$ on $\Tb^n$,
a nodal domain $\Al_\lambda$ and a ball $B$ of radius $\sim 1$
such that
$$\frac{\Vol(\Al_\lambda\cap B)}{\Vol(B)} \leq
\frac{C_3(n)}{(\sqrt{\lambda})^{n-1}}\ ,$$ where
$\Al_\lambda\not\subseteq B$ and
$\Al_\lambda\cap\Bhf\neq\emptyset$.
\end{thm}
\begin{proof}
Let $\Tb^n = \Rb^n/\ZZ^n$ be the standard flat torus parametrized
by the standard coordinates $(x_1, x_2, \ldots x_n)$, where $0\leq
x_j<1$. Let $k$ be an integer, and let 
$\vphi_k = \Pi_{j=1}^{n-1}\sin 2\pi k x_j$. $\vphi_k$ is an
eigenfunction corresponding to the eigenvalue $\lambda_k=
4(n-1)k^2\pi^2$. Each nodal domain has cross sections in normal direction 
to the $x_n$-axis of area $<c/k^{n-1}$.
Hence, if we take a ball~$B\subseteq\Tb^n$ of
radius~$1$ and we let $\Al_\lambda$ be a nodal domain which
contains the center of the ball, we have
$$\frac{\Vol(\Al_\lambda\cap B)}{\Vol(B)}\leq
   \frac{C(n)}{(\sqrt{\lambda})^{n-1}}\ .$$
\end{proof}

\providecommand{\bysame}{\leavevmode\hbox to3em{\hrulefill}\thinspace}
\providecommand{\MR}{\relax\ifhmode\unskip\space\fi MR }
\providecommand{\MRhref}[2]{%
  \href{http://www.ams.org/mathscinet-getitem?mr=#1}{#2}
}
\providecommand{\href}[2]{#2}

\noindent Dan Mangoubi,\\ 
Einstein Institute of Mathematics,\\
Hebrew University, Givat Ram,\\
Jerusalem 91904,\\
Israel\\
%
\smallskip
\texttt{\small mangoubi@math.huji.ac.il}
\end{document}